\documentclass[informs-stsy]{arxiv}              % for a regular run

\usepackage{natbib}
\NatBibNumeric
\def\bibfont{\small}%
\bibpunct[, ]{[}{]}{,}{n}{}{,}%

\usepackage[hidelinks=true,bookmarksopen=false,draft=false]{hyperref}

\usepackage[utf8]{inputenc}
\usepackage[T1]{fontenc}
\usepackage[english]{babel}
\usepackage{enumitem}
\usepackage[dvipsnames]{xcolor}

\usepackage{xfrac}
\usepackage{cancel}
\usepackage{bbm}

\usepackage{subcaption}
\usepackage{multirow}

\def\EMAIL#1{\href{mailto:#1}{#1}}% When hyperref is used, otherwise outcomment
\def\URL#1{\href{#1}{#1}}         % When hyperref is used, otherwise outcomment

\TheoremsNumberedThrough     % Preferred (Theorem 1, Lemma 1, Theorem 2)
\EquationsNumberedThrough    % Default: (1), (2), ...
\newcommand{\numberthis}{\addtocounter{equation}{1}\tag{\theequation}}

\newcommand{\tth}{{}^{\text{th}}}
\newcommand{\defn}{\stackrel{\triangle}{=}}
\newcommand{\dto}{\downarrow}

\newcommand{\amax}{A_{\max}}
\newcommand{\smax}{S_{\max}}

\newcommand{\bE}{\mathbb{E}}

\newcommand{\bP}{\mathbb{P}}
\newcommand{\bR}{\mathbb{R}}
\newcommand{\bZ}{\mathbb{Z}}

\newcommand{\cF}{\mathcal{F}}

\newcommand{\cX}{\mathcal{X}}

\newcommand{\va}{\boldsymbol{a}}

\newcommand{\ve}{\boldsymbol{e}}

\newcommand{\vq}{\boldsymbol{q}}
\renewcommand{\vs}{\boldsymbol{s}}
\newcommand{\vu}{\boldsymbol{u}}

\newcommand{\vx}{\boldsymbol{x}}
\newcommand{\vy}{\boldsymbol{y}}

\newcommand{\vone}{\boldsymbol{1}}
\newcommand{\vzero}{\boldsymbol{0}}

\newcommand{\qbar}{\overline{q}}
\newcommand{\abar}{\overline{a}}
\newcommand{\sbar}{\overline{s}}
\newcommand{\ubar}{\overline{u}}

\newcommand{\vqbar}{\overline{\vq}}
\newcommand{\vabar}{\overline{\va}}
\newcommand{\vsbar}{\overline{\vs}}
\newcommand{\vubar}{\overline{\vu}}

\newcommand{\Xbar}{\overline{X}}

\newcommand{\E}[1]{\bE\left[#1 \right]}
\newcommand{\Var}[1]{\text{Var}\left[#1 \right]}

\newcommand{\Prob}[1]{\bP\left[#1\right]}
\newcommand{\ind}[1]{\mathbbm{1}_{\left\{#1 \right\}}}
\newcommand{\expo}[1]{\exp\left({\scriptstyle #1}\right)}

\newcommand{\Eq}[1]{\bE_{\vq}\left[#1 \right]}

\newcommand{\pN}{{(N)}}
\newcommand{\thetahat}{\tilde{\theta}}

\makeatletter
\newcommand{\pushright}[1]{\ifmeasuring@#1\else\omit\hfill$\displaystyle#1$\fi\ignorespaces}

\allowdisplaybreaks

\begin{document}

\TITLE{Load balancing under Join the Shortest Queue: Many-Server Heavy-Traffic Asymptotics}

\ARTICLEAUTHORS{%
	\AUTHOR{Daniela Hurtado-Lange}
	\AFF{Georgia Institute of Technology, \EMAIL{d.hurtado@gatech.edu}, \URL{https://sites.google.com/view/daniela-hurtado-lange}}
	\AUTHOR{Siva Theja Maguluri}
	\AFF{Georgia Institute of Technology, \EMAIL{siva.theja@gatech.edu}, \URL{https://sites.google.com/site/sivatheja/}} }

\ABSTRACT{
We study the load balancing system operating under Join the Shortest Queue (JSQ) in the many-server heavy-traffic regime. If $N$ is the number of servers, we let the difference between the total service rate and the total arrival rate be $N^{1-\alpha}$ with $\alpha>0$. We show that for $\alpha>4$ the average queue length behaves similarly to the classical heavy-traffic regime. Specifically, we prove that the distribution of the average queue length multiplied by $N^{1-\alpha}$ converges to an exponential random variable. Moreover, we show a result analogous to state space collapse. We provide two proofs for our result: one using the one-sided Laplace transform, and one using Stein's method. We additionally obtain the rate of convergence in the Wasserstein's distance. }


\maketitle

\section{Introduction}

\label{sec:introduction}

The load balancing system, also known as the supermarket checkout system, is a parallel-server stochastic processing network (SPN), where each server has a separate queue. There is a single stream of arrivals that are routed to the queues by a single dispatcher. This model has been extensively studied since the '70s, when the celebrated join the shortest queue (JSQ) routing policy was introduced \cite{winston_JSQ_1977}. Under JSQ, the new arrivals are routed to the server with the least number of jobs (including the one in service, if any). It has been proved that JSQ minimizes delay under different notions of optimality in both, continuous and discrete time settings \cite{winston_JSQ_1977,weber1978optimal,ephremides1980simple,JSQ_HT_optimality,atilla}. However, understanding the probabilistic behavior of the queue lengths in steady state usually becomes intractable. Then, a common practice is to study the system under some asymptotic regime.

A popular regime is heavy traffic, where the number of servers is fixed and the arrival rate is increased to the maximum capacity. One of the advantages of this regime is that many systems behave as lower-dimensional systems in the limit, a phenomenon known as state space collapse (SSC). In particular, the load balancing system operating under JSQ experiences SSC into a one-dimensional subspace and, therefore, it behaves as a single-server queue in the limit. Heavy-traffic analysis of the load balancing system operating under JSQ has been performed in the past, and it has been proved that the vector of queue lengths can be well-approximated by a vector of the form $\Upsilon \vone$, where $\Upsilon$ is an exponential random variable. There are multiple proofs for this result, including the diffusion limits approach \cite{JSQ_HT_optimality}, the drift method \cite{atilla} and transform techniques based on the drift method \cite{Hurtado_transform_method}. 

Another popular asymptotic regime is mean field, where the number of servers increases to infinity while the load is kept constant. This regime has been used in the past to study the load balancing system operating under power-of-$d$ choices (also called JSQ$(d)$), under which $d$ queues are selected at random and the new arrivals are routed to the shortest queue among these. It has been proved that, even if $d=2$, the queue lengths decrease considerably when compared to random routing \cite{mitzenmacher_po2,mitzenmacher_po2_2,dobrushin_po2}. To prove this result, the main idea is that, since only $d$ queues are being compared whenever there is an arrival, and $d$ is small compared to the number of servers, in the mean-field limit the queues are almost independent. However, this argument does not work if we use JSQ for routing.

Regimes in between these two are when the load and the number of servers increase together, and there are several ways to take this type of limit. We call these regimes many-server heavy-traffic. In this paper, we focus on the case where the difference between the \emph{total} service rate and the \emph{total} arrival rate is of the form $N^{1-\alpha}$, where $N$ is the number of servers and $\alpha>0$ is a parameter. Equivalently, we can define the difference between service and arrival rates \emph{per server} as $N^{-\alpha}$. Then, as $N\to\infty$ the number of servers and the load increase. 

Depending on the value of $\alpha$, the behavior of the load balancing system is different. First observe that if $\alpha\to\infty$, we are heuristically approaching the classical heavy-traffic regime. Similarly, if $\alpha\dto 0$ we heuristically approach the mean-field regime. There are known phase transitions at $\alpha=1$ and $\alpha=\frac{1}{2}$. If $\alpha=1$ we are in the Nondegenerate Slowdown (NDS) regime \cite{Atar_NDS}, if $\alpha=\frac{1}{2}$ we are in the Halfin-Whitt regime \cite{HalfinWhitt_Regime} and if $\alpha<\frac{1}{2}$ we are in sub-Halfin-Whitt regime. Load balancing systems in the many-server heavy-traffic limit have been extensively studied in the literature(see for example the survey paper \cite{load_balancing_survey}). However, the study of load balancing systems operating under JSQ in the many-server heavy-traffic limit started recently, and it has gained plenty of attention \cite{Gamarnik_JSQ,Braverman2020_jsq,BanMukh_JSQ_tail_asymptotics,BanMukh_JSQ_sensitivity,liu2018simple,LiuLei_JSQ_UniversalScaling,GupWal_NDS_JSQ}. In \cite{Gamarnik_JSQ} a load balancing system under JSQ is studied for the first time in the many-server heavy-traffic limit. The authors show that, in the Halfin-Whitt regime, the number of empty queues and the number of queues with one customer in line are of order $O(\sqrt{N})$. The authors use the diffusion limits approach, but interchange of limits is not proved. This step is completed in \cite{Braverman2020_jsq}. In \cite{BanMukh_JSQ_tail_asymptotics,BanMukh_JSQ_sensitivity} the work of \cite{Gamarnik_JSQ} is continued. Specifically, in \cite{BanMukh_JSQ_tail_asymptotics} the authors study tail asymptotics of the stationary distribution, and in \cite{BanMukh_JSQ_sensitivity} they study the moments of the stationary distribution. In \cite{liu2018simple} load balancing systems under several routing policies are studied in the sub-Halfin-Whitt regime, and in \cite{LiuLei_JSQ_UniversalScaling} the analysis is extended to the case when $\alpha\in\left[\frac{1}{2},1\right)$. In \cite{Weina_pod_zero_delay} a load balancing system operating under power-of-$d$, where jobs are batches of tasks, is analyzed. Specifically, the authors find conditions on the value of $d$ (as a function of the number of servers, the load and the number of tasks per job) such that power-of-$d$ choices achieves zero delay in sub-Halfin-Whitt regime. In \cite{GupWal_NDS_JSQ}, load balancing systems are studied in the NDS regime. A summary of these results depending on the value of $\alpha$ is presented in Table \ref{future.work.table:jsq.alpha}.

\begin{table}
	\centering
	\caption{Literature review for asymptotic regimes depending on the value of $\alpha$.}\label{future.work.table:jsq.alpha}
	\begin{tabular}{|l|l|l|}
		\hline
		Value of $\alpha$ & Regime & References \\ \hline
		$\alpha\dto 0$ & Mean-field & \\ \hline
		$\alpha\in\left(0,\frac{1}{2}\right)\vphantom{\dfrac{num}{den}}$ & Sub-Halfin-Whitt & \cite{liu2018simple,Weina_pod_zero_delay} \\ \hline
		$\alpha=\frac{1}{2}\vphantom{\dfrac{num}{den}}$                  & Halfin-Whitt & \cite{Gamarnik_JSQ,Braverman2020_jsq,BanMukh_JSQ_tail_asymptotics,BanMukh_JSQ_sensitivity} \\ \hline
		$\alpha\in\left(\frac{1}{2},1\right)\vphantom{\dfrac{num}{den}}$ & & \cite{LiuLei_JSQ_UniversalScaling} \\ \hline
		$\alpha=1$  & Nondegenerate Slowdown (NDS) & \cite{GupWal_NDS_JSQ} \\ \hline
		$\alpha\in(1,2]$                 &  Future work & \\ \hline
		$\alpha\in(2,\infty)$                 &  This paper & \\ \hline
		$\alpha\to\infty$ & Classical heavy-traffic & \cite{atilla} \\ \hline
	\end{tabular}
\end{table}

Our main contribution is to show that if $\alpha>4$ the behavior of the average queue length in the load balancing system operating under JSQ in the many-server heavy-traffic regime is similar to its behavior in the classical heavy-traffic regime. Specifically, we prove that the distribution of the average queue length is exponential (Theorem \ref{thm}), which is the same distribution as in heavy-traffic regime. We also compute the rate of convergence in Wasserstein's distance (Theorem \ref{thm:Stein}).
Moreover, we prove a result that is analogous to SSC in the classical heavy-traffic regime (Proposition \ref{prop:ssc}) and, consequently, we also call it SSC. Specifically, our SSC result shows that the number of jobs in line is similar for all the queues as $N$ gets large. We prove this result by showing that the error of approximating the actual queue length by the average queue length becomes negligible in the limit.

Our goal in this paper is to collaborate to closing the gap between the NDS and the classical heavy-traffic regimes. We find the distribution of the total queue lengths for $\alpha>4$, but there is still a gap between our result and NDS regime, that is, for $\alpha\in(1,4]$. Exploring this gap is future work.

To the best of our knowledge, we are the first ones to study the load balancing system operating under JSQ under this regime. We provide two proofs of the theorem: one using the transform method introduced in \cite{Hurtado_transform_method}, and one using Stein's method \cite{Ross2011_SteinSurvey,gurvich2014diffusion,braverman2017stein,braverman2017stein2,Walton_SteinHT}. We briefly describe each of the methods below.

The transform method introduced in \cite{Hurtado_transform_method} is a two-step procedure to compute the distribution of the queue lengths in classical heavy-traffic regime, and it can be used for queueing systems that experience SSC to a one-dimensional subspace. The method is introduced in the context of a load balancing system and a generalized switch. Before using the method, positive recurrence and SSC to a one-dimensional subspace must be proved. The main idea is to consider an exponential test function such that, after setting its drift to zero, yields the Moment Generating Function (MGF) of the projection of the vector of the queue lengths on the subspace where SSC occurs. Then, an implicit expression that is valid for all traffic is obtained. The last step is to take the heavy-traffic limit, and prove that the terms depending on the queue lengths vanish, so that we obtain an explicit expression for the limiting MGF. The authors in \cite{Hurtado_transform_method} work with the MGF, but the transform method they propose also works with other transforms, such as the one-sided Laplace transform or the characteristic function (also known as Fourier transform). In this paper we use the one-sided Laplace transform.

Stein's method is based on the approach introduced by \cite{stein1972bound}, and a survey about the main results can be found in \cite{Ross2011_SteinSurvey}. The main idea is to bound the Wasserstein's distance between the pre-limit variable and a random variable that follows the limiting distribution. Then, one obtains conditions on the parameters of the system such that this bound converges to zero. In \cite{gurvich2014diffusion} the method was developed for continuous time Markov chains, and it was first used in the context of SPNs in \cite{braverman2017stein,braverman2017stein2}. The method was developed there for single queues with many servers, with patient and impatient customers. Later, in \cite{Walton_SteinHT} the same method was used to study a single server queue in heavy-traffic. We use it here to bound the Wasserstein's distance between the total queue lengths and an exponential random variable.

The organization of this paper is as follows. In Section \ref{sec:model} we provide the details of the load balancing model and JSQ routing, and we state the main result of this paper. In Section \ref{sec:SSC} we formally show the notion of SSC described above. Then, in Section \ref{sec:MGF} we prove the main result using the transform methods proposed in \cite{Hurtado_transform_method} applied to the one-sided Laplace transform;  and in Section \ref{sec:Stein} we prove the result using Stein's method and we additionally provide the rate of convergence of the average queue length to the corresponding exponential random variable.

\subsection{Notation}

We use $\bR$ and $\bZ$ to denote the sets of real and integer numbers, respectively. We add a subscript $+$ when we refer to nonnegative numbers, and a superscript $n\in\bZ_+$ when we mean vector spaces. Given $n\in\bZ_+$, we use $[n]\defn \left\{i\in\bZ_+: i\leq n \right\}$. 
We use bold letters to denote vectors. Given two vectors $\vx,\vy\in\bR^n$, we use $\langle\vx,\vy\rangle$ to denote their dot product, and for $p\in\bZ_+$ with $p\geq 1$ we use $\|\vx\|_p$ to denote the $p$-norm of $\vx$. We use $\vone$ and $\vzero$ to denote the vectors of all ones and all zero elements, respectively. For $n\in\bZ_+$ and $i\in[n]$, we use $\ve^{(i)}$ to denote the $n$-dimensional $i\tth$ canonical vector, i.e., an $n$-dimensional vector with 1 in the $i\tth$ element and 0 everywhere else. 

Given a random variable $X$ we use $\E{X}$ to denote its expected value and $\Var{X}$ for its variance. For an event $A$ we use $\ind{A}$ to denote the indicator function of $A$. Given a random process $\{\vq(k): k\geq 1\}$ (that will be later defined as the queue lengths process), we use $\Eq{\cdot}\defn \E{\cdot | \vq(k)=\vq}$. We use $\Rightarrow$ to denote convergence in distribution. 

For a Markov chain $\{X(k):k\geq 1 \}$ with countable state space $\cX$ and a function $Z:\cX\to\bR_+$, define the drift of $Z$ at $x$ as
\begin{align*}
\Delta Z(x)\defn \left[Z(X(k+1))- Z(X(k)) \right]\ind{X(k)=x}.
\end{align*}
Thus, $\Delta Z(x)$ is a random variable that measures the amount of change in the value of $Z$ in one time slot.

For a function $f$ with domain $Dom(f)$ we denote $\|f\|\defn \sup_{x\in Dom(f)}|f(x)|$. Finally, we say that $f(N)$ is of order $o(g(N))$ if $\lim_{N\to\infty} \frac{f(N)}{g(N)}=0$.

\section{Model}\label{sec:model}

Consider a load balancing system, i.e., a queueing system with $N$ parallel servers, each of them with an infinite buffer. There is a single stream of arrivals to the system, and, upon arrival, the jobs are routed to the queues by a dispatcher. After being routed, the jobs cannot move from one queue to another, and they wait in line until the server processes them.

We model the system in discrete time (i.e., in a time-slotted fashion), and we index time by $k\in\bZ_+$. For each $k\in \bZ_+$ and $i\in[N]$, let $q_i(k)$ be the number of jobs in queue $i$ at the beginning of time slot $k$ (including the job in service, if any). Let $\left\{a(k):k\in\bZ_+ \right\}$ be a sequence of i.i.d. random variables such that $a(k)$ is the number of arrivals to the system in time slot $k$. Let $\va(k)$ be a vector where $a_i(k)$ is the number of jobs that are routed to queue $i$ in time slot $k$, for $i\in[N]$.  We assume that all the arrivals in one time slot are routed to the same queue. Then, if the dispatcher routes the arrivals to queue $i^*$, we have $\va(k)= a(k)\ve^{(i^*)}$. Let $\left\{\vs(k):k\in\bZ_+ \right\}$ be a sequence of i.i.d. random vectors such that $s_i(k)$ is the potential number of jobs that can be processed by server $i$ in time slot $k$. If $s_i(k)$ is larger than the number of jobs in queue $i$ (including the arrivals in time slot $k$), then there is unused service, that we denote by $u_i(k)$. Assume that the sequences $\{a(k):k\in\bZ_+\}$ and $\{s_i(k):k\in\bZ_+ \}$ for each $i\in[N]$ are independent of each other, and of the queue length process $\{\vq(k):k\in\bZ_+\}$. Additionally, let $\amax$ and $\smax$ be finite constants such that $a(1)\leq N\amax$ and $s_i(1)\leq \smax$ for all $i\in[N]$ with probability 1.

Routing occurs according to JSQ. Then, the arrivals in time slot $k$ are routed to the $i^*\tth$ queue, where
\begin{align*}
i^*(k)\in \argmin_{i\in[N]} q_i(k),
\end{align*}
breaking ties at random.

In each time slot, the order of events is as follows. First, the queue lengths are observed, second, arrivals occur, third, arrivals are routed according to JSQ and, at the end of each time slot, service occurs. Then, the dynamics of the queues are described by the equation
\begin{align}\label{eq:dynamics.max}
q_i(k+1) &= \max\left\{q_i(k)+a_i(k)-s_i(k),0 \right\}\quad\forall i\in[N],\;\forall k\in\bZ_+,
\end{align}
which, by definition of unused service, is equivalent to 
\begin{align}\label{eq:dynamics.u}
q_i(k+1) &= q_i(k)+a_i(k)-s_i(k)+u_i(k)\quad\forall i\in[N],\;\forall k\in\bZ_+.
\end{align}
From the dynamics of the queues, it is easy to prove that
\begin{align}\label{eq:qu}
q_i(k+1)u_i(k)=0 \quad\forall i\in[N],\;\forall k\in\bZ_+
\end{align}
with probability 1. Intuitively, if there is unused service in queue $i$ in time slot $k$ ($u_i(k)>0$), the potential service is larger than the number of jobs in line available to be served (i.e., $q_i(k)+a_i(k)<s_i(k)$). Then, the queue should be empty in time slot $k+1$. Therefore, $u_i(k)>0$ implies $q_i(k+1)=0$. Similarly, $q_i(k+1)>0$ implies $u_i(k)=0$.

For each $i\in[N]$, assume $\E{s_i(1)}=1$ and $\Var{s_i(1)}=\sigma_s^2$. We are interested in the many-server heavy-traffic limit, so we parametrize the system by the number of servers $N$ in the following way. We add a superscript $(N)$ to the variables when we refer to the parametrized system. Let $\lambda^\pN \defn \E{a^\pN(1)} = N\left(1-N^{-\alpha}\right)$, where $\alpha>0$ and $\Var{a^\pN(1)}=N\sigma_a^2$. Observe that
\begin{align}\label{eq:epsilon}
\sum_{i=1}^N \E{s_i(1)}- \E{a^\pN(1)} = N^{1-\alpha},
\end{align}
which is positive. Therefore, the Markov Chain $\left\{\vq^\pN(k):k\in\bZ_+ \right\}$ is positive recurrent \cite[Lemma 2]{atilla}. Assume $\Prob{a(k)-s_i(k)=0}>0$ for some $i\in[N]$. Then, $\left\{\vq^\pN(k):k\in\bZ_+ \right\}$ is also aperiodic because  $\Prob{\vq(k)=\vq(k+1)}>0$. Then, the vector of queue lengths converges in distribution to a steady-state random vector, that we denote $\vqbar^\pN$. Let $\abar^\pN$ be a steady-state random variable with the same distribution as $a^\pN(1)$ and $\vsbar$ be a steady-state random vector with the same distribution as $\vs(1)$. Let $\vabar^\pN$ be the vector of arrivals after routing in steady state, given that the queue lengths are $\vqbar^\pN$ and $\abar^\pN$ jobs arrive to the system, and let $\vubar^\pN$ be the vector of unused service in steady-state, given $\vqbar^\pN$, $\abar^\pN$ and $\vsbar^\pN$. Define $\left(\vqbar^\pN\right)^+\defn \vqbar^\pN + \vabar^\pN -\vsbar + \vubar^\pN$ as the vector of queue lengths one time slot after $\vqbar^\pN$ is observed, given $\abar^\pN$ and $\vsbar$. Then, from \eqref{eq:qu} we have $\left(\qbar_i^\pN\right)^+\ubar^\pN_i=0$ with probability 1 for all $i\in[N]$.

In the next sections we prove Theorem \ref{thm}, using two different approaches.

\begin{theorem}\label{thm}
	Consider a sequence of load balancing systems operating under JSQ, parametrized by $N$ as described above. If $\alpha>4$, then $ N^{-\alpha}\sum_{i=1}^N\qbar_i^\pN\Rightarrow \Upsilon$ as $N\to\infty$, where $\Upsilon$ is an exponential random variable with mean $\frac{\sigma_a^2+\sigma_s^2}{2}$.
\end{theorem}

Similarly to the classical heavy-traffic regime \cite{atilla,Hurtado_transform_method}, proving SSC and bounding the unused service are essential in the proof of Theorem \ref{thm}. Further computing the expected total unused service in steady state is essential in the proof of SSC and in the proof of Theorem \ref{thm}. We state the result below, as it will be repeatedly in the rest of this paper. 
\begin{lemma}\label{lemma:Eu}
	Consider a load balancing system operating under JSQ, parametrized by $N$ as described in Section \ref{sec:model}. Then,
	\begin{align*}
	\E{\sum_{i=1}^N \ubar_i^\pN} = N^{1-\alpha}.
	\end{align*}
\end{lemma}
\proof{Proof of Lemma \ref{lemma:Eu}.}
	In this proof we omit the dependence on $N$ of the variables for ease of exposition. We set to zero the drift of $V_\ell(\vq)=\sum_{i=1}^N q_i$. Before doing it, we should show that $\E{V_\ell(\vqbar)}<\infty$. This result is a direct consequence of \cite[Proposition 3]{atilla}, so we omit the proof. Setting the drift to zero we obtain
	\begin{align*}
	0&= \E{V_\ell(\vqbar^+)-V_\ell(\vqbar)} \\
	&= \E{\sum_{i=1}^N \left(\qbar_i^+ - \qbar_i\right)} \\
	&\stackrel{(a)}{=} \E{\sum_{i=1}^N \left(\abar_i-\sbar_i+\ubar_i \right)},
	\end{align*}
	where $(a)$ holds by definition of $\vqbar^+$. Rearranging terms we obtain
	\begin{align*}
	\E{\sum_{i=1}^N \ubar_i}& = \E{\sum_{i=1}^N \left(\sbar_i - \abar_i\right)} \\
	&\stackrel{(a)}{=} \sum_{i=1}^N \E{\sbar_i} - \E{\abar} \\
	&\stackrel{(b)}{=} N^{1-\alpha},
	\end{align*}
	where $(a)$ holds because $\abar=\sum_{i=1}^N \abar_i$ by definition; and $(b)$ holds by \eqref{eq:epsilon}.  
\Halmos \endproof

Observe that in the proof of Lemma \ref{lemma:Eu} we do not specifically use that the routing policy is JSQ. In fact, all we need is a routing policy such that $\E{\sum_{i=1}^N \qbar_i^\pN}<\infty$, and this condition is usually satisfied by throughput optimal policies. For example, power-of-$d$ choices satisfies this condition as well.

%%%%%%%%%%%%%%%%%%%%%%%%%%%%%%%%%%%%%%%%%%%%%%%%%%%%%%%%%%%%%%%%%%%%%%%%%%%%%%%%%%%%%%%%%%%%
%%%%%%%%%%%%%%%%%%%%%%%%%%%%%%%%%%%%%%%%%%%%%%%%%%%%%%%%%%%%%%%%%%%%%%%%%%%%%%%%%%%%%%%%%%%%
%%%%%%%%%%%%%%%%%%%%%%%%%%%%%%%%%%%%%%%%%%%%%%%%%%%%%%%%%%%%%%%%%%%%%%%%%%%%%%%%%%%%%%%%%%%%

\section{State Space Collapse}\label{sec:SSC}

The goal of this paper is to show that when $\alpha>4$, the average queue length in the load balancing system in the many-server heavy-traffic limit behaves similarly to the classical heavy-traffic regime. It is known that the load balancing system operating under JSQ exhibits one-dimensional SSC in classical heavy traffic, i.e., in the limit it behaves as a single-server queue. In \cite{atilla}, SSC is proved by showing that the error of approximating the actual vector of queue lengths by its projection on the line where SSC occurs is bounded, and the bound does not depend on the traffic intensity. Therefore, as the traffic intensity increases, this error is negligible. In our case, the traffic intensity depends on the number of servers, so we need to show that the bound becomes negligible as $N$ increases. 
Before stating the result we introduce notation. Given a vector $\vx\in\bR^N$, let 
\begin{align}\label{eq:def-parallel}
\vx_\parallel \defn \left(\sum_{i=1}^N\frac{x_i}{N}\right)\vone\quad \text{and}\quad \vx_\perp \defn \vx-\vx_\parallel.
\end{align}
Then, $\vx_\parallel$ is the projection of $\vx$ on the line generated by the vector $\vone$, and $\vx_\perp$ is the error of approximating $\vx$ by $\vx_\parallel$. In this section we prove the following proposition.

\begin{proposition}\label{prop:ssc}
	Consider a load balancing system operating under JSQ, parametrized by $N$ as described in Section \ref{sec:model}. Let $\delta\in(0,1)$ and $N_0\in\bZ_+$ be such that $\delta\leq 1-N^{-\alpha}$ for all $N\geq N_0$. Then, there exists a finite constant $C$ such that for any $r\in\bZ_+$ with $r\geq 1$, we have
	\begin{align}\label{eq:SSC.moments}
		\E{\left\|\vqbar_\perp^\pN \right\|^r}^{\frac{1}{r}}\leq C r N^3.
	\end{align}
	Additionally, if $\theta\leq \frac{1}{4\amax N^{\frac{5}{2}-\alpha}}\log\left(1 + \frac{\delta}{4\amax N^{\frac{3}{2}}}\right)$ we have
	\begin{align}\label{eq:SSC.mgf}
		\E{\expo{\theta N^{1-\alpha}\left\|\vqbar_\perp^\pN \right\|}}\leq \dfrac{\delta \expo{\frac{2\theta N^{3-\alpha}}{\delta}} }{\delta + 4\amax N^{\frac{3}{2}}\left(1 - \expo{4\theta\amax N^{\frac{5}{2}-\alpha}} \right) }.
	\end{align}
\end{proposition}

In the proof of Proposition \ref{prop:ssc} we use the moment and exponential bounds based on drift arguments proved in \cite[Lemma 3]{MagSri_SSY16_Switch}. We present the proof in Appendix \ref{app:bertsimas_lemma} for completeness.%The proof of Proposition \ref{prop:ssc} is based on the following Lemma, which was proved in \cite{MagSri_SSY16_Switch}. We state it below for completeness.

%%%%%%%%%%%%%%%%%%%%%%%%%%%%%%%%%%%%%%%%%%%%%%%%%%%%%%%%%%%%%%%%%%%%%%%%%%%%%%%%%%%%%%%%%%%%%%%%%%%%%%%%%%%%%%%%%%%%%%%%%%%%%%%%%%%%%%%%%%%%%%%%%%%%%%%%%%%%%%%%%%%%%%%%%%%%%%%%%%%%%%%%%
\section{Proof of Theorem \ref{thm} using transform methods}\label{sec:MGF}

In this section we prove Theorem \ref{thm} using the Transform method introduced in \cite{Hurtado_transform_method} applied to the one-sided Laplace transform. 

\proof{Proof of Theorem \ref{thm} using one-sided Laplace transform.}
	For ease of exposition, in this proof we omit the dependence on $N$ of the variables. 
	The first step is to prove an `exponential version' of \eqref{eq:qu}. We state the result below.
	
	\begin{lemma}\label{lemma:expo_qu}
		Consider a load balancing system operating under JSQ, parametrized by $N$ as described in Section \ref{sec:model} and suppose $\alpha>4$. Let $\thetahat\defn \theta\left(\frac{\sigma_a^2+\sigma_s^2}{2}\right)$, where $\theta<0$ is a finite parameter. Then, there exists $\tilde{\Theta}>0$ such that for all $|\thetahat|<\tilde{\Theta}$ we have
		\begin{align*}
		\left|\E{\left(\expo{\thetahat N^{-\alpha} \sum_{i=1}^N \left(\qbar_i^\pN\right)^+} -1\right)\left(\expo{-\thetahat N^{-\alpha}\sum_{i=1}^N \ubar_i^\pN} -1\right)} \right|\;\text{ is of order $o\left(N^{1-2\alpha}\right)$},
		\end{align*}
		where $\qbar_\Sigma^+$ represents the total queue length one time slot after observing $\qbar_\Sigma$.
	\end{lemma}
	
	We present the proof of Lemma \ref{lemma:expo_qu} in Appendix \ref{app:lemma-expo-qu}. Rearranging terms in the expression of Lemma \ref{lemma:expo_qu} we obtain
		\begin{align*}
	& \E{\expo{\thetahat N^{-\alpha}\sum_{i=1}^N \qbar_i^+}} - \E{\expo{\thetahat N^{-\alpha}\sum_{i=1}^N \qbar_i}}\E{\expo{\thetahat N^{-\alpha}\left(\abar-\sum_{i=1}^N\sbar_i\right)}} \\
	&= 1-\E{\expo{-\thetahat N^{-\alpha}\sum_{i=1}^N \ubar_i}} + o\left(N^{1-2\alpha}\right),
	\end{align*}
	where we used the dynamics of the queues \eqref{eq:dynamics.u}, the fact that $\abar=\sum_{i=1}^N \abar_i$, and that the arrival process to the system and the potential service processes are independent of the queue lengths.
	
	Since $\thetahat<0$, we know $\E{\expo{\thetahat N^{-\alpha}\sum_{i=1}^N \qbar_i}}\leq 1<\infty$. Then, we can set its drift to zero, i.e., we set $\E{\expo{\thetahat N^{-\alpha}\sum_{i=1}^N \qbar_i}} =\E{\expo{\thetahat N^{-\alpha}\sum_{i=1}^N \qbar_i^+}}$. Using this property and rearranging terms we obtain
	\begin{align}\label{eq:thm_fraction}
	\E{\expo{\thetahat N^{-\alpha}\sum_{i=1}^N \qbar_i}} = \dfrac{ 1-\E{\expo{-\thetahat N^{-\alpha}\sum_{i=1}^N \ubar_i}} + o\left(N^{1-2\alpha}\right)}{1- \E{\expo{\thetahat N^{-\alpha}\left(\abar-\sum_{i=1}^N \sbar_i\right)}} }
	\end{align}

	Our goal is to take the limit as $N\to\infty$. Observe that the right-hand side of \eqref{eq:thm_fraction} yields a $\frac{0}{0}$ form in the limit. Then, we take the Taylor expansion with respect to $\thetahat$ of the numerator and denominator. For the numerator we obtain
	\begin{align*}
	1-\E{\expo{-\thetahat N^{-\alpha}\sum_{i=1}^N \ubar_i}} & = \thetahat N^{-\alpha} \E{\sum_{i=1}^N \ubar_i} + o\left(N^{1-2\alpha}\right) = \thetahat N^{1-2\alpha} + o\left(N^{1-2\alpha}\right),
	\end{align*}
	where the last equality holds by Lemma \ref{lemma:Eu}. The $o\left(N^{1-2\alpha}\right)$ term arises because for all $r\geq 2$ we have
	\begin{align*}
	\left|\dfrac{\thetahat^r N^{-r\alpha}}{r!}\E{\left(\sum_{i=1}^N \ubar_i \right)^r} \right| & \stackrel{(a)}{=} \dfrac{|\thetahat|^r N^{-r\alpha}}{r!}\E{\left(\sum_{i=1}^N \ubar_i \right)^{r-1} \left(\sum_{i=1}^N \ubar_i\right)} \\
	& \stackrel{(b)}{\leq} \frac{|\thetahat|^r \smax^{r-1}}{r!}N^{r(1-\alpha)-1}\E{\sum_{i=1}^N \ubar_i} \\
	& \stackrel{(c)}{=} \frac{|\thetahat|^r \smax^{r-1}}{r!}N^{r-\alpha(r+1)},
	\end{align*}
	where $(a)$ holds because $\ubar_i\geq 0$ with probability 1 for all $i\in[N]$ by definition of unused service; $(b)$ holds because $0\leq\ubar_i\leq\sbar_i\leq\smax$; and $(c)$ holds by Lemma \ref{lemma:Eu}. Also, $r-\alpha(r+1)-(1-2\alpha) = (r-1)(1-\alpha)$, which is negative for all $\alpha>1$. Then, it is negative for $\alpha>4$.
	
	For the denominator we obtain
	\begin{align*}
	& 1- \E{\expo{\thetahat N^{-\alpha}\left(\abar-\sum_{i=1}^N \sbar_i\right)}} \\
	& = -\thetahat N^{-\alpha}\E{\abar-\sum_{i=1}^N \sbar_i} - \frac{\thetahat^2 N^{-2\alpha}}{2}\E{\left(\abar-\sum_{i=1}^N \sbar_i\right)^2} + o\left(N^{1-2\alpha}\right) \\
	& \stackrel{(a)}{=} \thetahat N^{1-2\alpha} - \dfrac{\thetahat^2 N^{1-2\alpha}}{2}\left(\sigma_a^2+\sigma_s^2\right) - \dfrac{\thetahat^2 N^{2-4\alpha}}{2}+o\left(N^{1-2\alpha}\right) ,
	\end{align*}
	where $(a)$ holds by definition of variance and because $\E{\sum_{i=1}^N \sbar_i-\abar}=N^{1-\alpha}$. The $o\left (N^{1-2\alpha} \right)$ arises similarly to the case of the numerator. We omit the details for brevity.
	
	Putting everything together and canceling $\thetahat N^{1-2\alpha}$ from the numerator and the denominator we obtain
	\begin{align*}
	\E{\expo{\thetahat N^{-\alpha}\sum_{i=1}^N \qbar_i}} & = \dfrac{1+o(1)}{1-\thetahat\left(\dfrac{\sigma_a^2+\sigma_s^2}{2}\right)+o(1)} \\
	& = \dfrac{1+o(1)}{1-\theta+o(1)},
	\end{align*}
	where the last equality holds because $\thetahat=\frac{2\theta}{\sigma_a^2+\sigma_s^2} $. Therefore,
	\begin{align*}
	\lim_{N\to\infty} \E{\expo{\thetahat N^{-\alpha}\sum_{i=1}^N \qbar_i}} &= \dfrac{1}{1-\theta},
	\end{align*}
	which is the one-sided Laplace transform of an exponential random variable with mean 1.   %This implies that $N^{-\alpha}\sum_{i=1}^N \qbar_i\Rightarrow \Upsilon$, where $\Upsilon$ is exponential with mean $\frac{\sigma_a^2+\sigma_s^2}{2}$.
\Halmos \endproof

%%%%%%%%%%%%%%%%%%%%%%%%%%%%%%%%%%%%%%%%%%%%%%%%%%%%%%%%%%%%%%%%%%%%%%%%%%%%%%%%%%%%%%%%%%%%%
%%%%%%%%%%%%%%%%%%%%%%%%%%%%%%%%%%%%%%%%%%%%%%%%%%%%%%%%%%%%%%%%%%%%%%%%%%%%%%%%%%%%%%%%%%%%%
%%%%%%%%%%%%%%%%%%%%%%%%%%%%%%%%%%%%%%%%%%%%%%%%%%%%%%%%%%%%%%%%%%%%%%%%%%%%%%%%%%%%%%%%%%%%%

\section{Proof of Theorem \ref{thm} using Stein's method}\label{sec:Stein}

The proof we provide in this section is based on bounding the Wasserstein's distance between the average queue length and an exponential random variable. We start with the definition of this metric as presented in \cite{Ross2011_SteinSurvey}.

\begin{definition}
	For two probability measures $\nu_1$ and $\nu_2$, the Wasserstein's distance between them is
	\begin{align*}
	d_W(\nu_1,\nu_2)\defn \sup_{h\in\emph{Lip(1)}} \left|\int h(x)\,d\nu_1(x) - \int h(x)\,d\nu_2(x) \right|,
	\end{align*}
	where $\emph{Lip(1)}=\left\{h:\bR\to\bR\text{ such that }|h(x)-h(y)|\leq |x-y| \right\}$ is the set of Lipschitz functions with constant 1.
\end{definition}

For random variables $X$ and $Y$ with laws $\nu_1$ and $\nu_2$, respectively, we write $d_W(X,Y)$ instead of $d_W(\nu_1,\nu_2)$, and when the measures are clear from the context we write
\begin{align}\label{eq:Wass-Eh}
d_W(X,Y) = \sup_{h\in\text{Lip(1)}} \left|\E{h(X)}-\E{h(Y)} \right|.
\end{align}

It is well-known that, if $\{X^{(n)}:n\in\bZ_+\}$ is a sequence of random variables, and $X$ is another random variable, then $d_W(X^{(n)},X)$ converging to zero as $n\to\infty$ implies that $X^{(n)}\Rightarrow X$ \cite{Ross2011_SteinSurvey}. We use this result to prove Theorem \ref{thm} as a consequence of the following theorem.

\begin{theorem}\label{thm:Stein}
	Consider a load balancing system operating under JSQ, parametrized by $N$ as described in Section \ref{sec:model}. Let $Z$ be an exponential random variable with mean 1. Then, we have
	\begin{align*}%\label{eq:Wass-dist-bound}
	&
	\begin{aligned}
		& d_W\left(\dfrac{2N^{-\alpha}}{\sigma_a^2+\sigma_s^2}\sum_{i=1}^N \qbar^\pN_i,Z \right) \\
			& \leq \dfrac{1}{\sigma_a^2+\sigma_s^2}\left(5\smax N^{1-\alpha} + N^{1-2\alpha} + C \smax(\alpha-1) N^{4-\alpha} \left\lceil\log\left(N\right)\right\rceil N^{\frac{1}{\left\lceil\log\left(N\right)\right\rceil}} \right. \\
			&\qquad \qquad \qquad \left. +  \dfrac{4N^{2-3\alpha}}{3(\sigma_a^2+\sigma_s^2)} \left(\amax + 2\smax\right)^3 \right).
	\end{aligned} 
	\end{align*}
	where $\overline{C}\defn C\smax e^{\frac{1}{2e}+1}$ and $C$ is the constant from Proposition \ref{prop:ssc}.
\end{theorem}
%The proof of Theorem \ref{thm} holds because one of the assumptions there is that $\alpha>4$. Then, we have an upper bound on $d_W\left(\frac{2N^{-\alpha}}{\sigma_a^2+\sigma_s^2}\sum_{i=1}^N \qbar^\pN_i,Z \right)$ that converges to zero as $N\to\infty$. 
Theorem \ref{thm} holds as an immediate consequence of Theorem \ref{thm:Stein} because when $\alpha>4$ (which is one of the assumptions of Theorem \ref{thm}), the upper bound converges to zero as $N\to\infty$.

Now we prove Theorem \ref{thm:Stein}. Our proof is based on the following lemma \cite[Theorem 5.4 part 1.]{Ross2011_SteinSurvey}.
\begin{lemma}\label{lemma:Wass-dist-Ross}
	Let $W\geq 0$ be a random variable with $\E{W}<\infty$, and let $Z$ be an exponential random variable with mean 1. Define
	\begin{align}\label{eq:Fk}
	\cF_W\defn \left\{g:\bR\to\bR \text{ such that } g(0)=0,\, \left\| g'\right\|\leq 1,\, \|g''\|\leq 2 \right\}.
	\end{align}
	Then,
	\begin{align*}
	d_W(W,Z)\leq \sup_{g\in \cF_W}\left|\E{g'(W)-g(W)} \right|.
	\end{align*}
\end{lemma}

Now we prove Theorem \ref{thm:Stein} using Stein's method.

\proof{Proof of Theorem \ref{thm:Stein}.}
	For ease of exposition, we omit the dependence on $N$ of the variables. We use Lemma \ref{lemma:Wass-dist-Ross} with 
	\begin{align*}
	W=q_\Sigma^\pN \defn  \dfrac{2N^{-\alpha}}{\sigma_a^2+\sigma_s^2}\sum_{i=1}^N \qbar_i.
	\end{align*}
	
	Let $g=f'\in\cF_W$. Observe that $f'\in\cF_W$ implies $f\in\text{Lip(1)}$ and, hence, $f$ is integrable. Therefore, if $f'\in\cF_W$, then $f$ is differentiable and well defined \cite[Theorem 7.2]{Wheeden2015_RealAnalysis}.
	
	Let $\qbar_\Sigma^+ \defn \left(\frac{2 N^{-\alpha}}{\sigma_a^2+\sigma_s^2}\right)\sum_{i=1}^N \qbar_i^+$. We start expanding $f\left(\qbar_\Sigma^+\right)$ in Taylor series around $\qbar_\Sigma$. There exists $\xi$ between $\qbar_\Sigma^+$ and $\qbar_\Sigma$ such that
	\begin{align}
	f\left(\qbar_\Sigma^+\right) & = f\left(\qbar_\Sigma\right) + \left(\qbar_\Sigma^+ - \qbar_\Sigma\right)f'\left(\qbar_\Sigma\right) + \frac{\left(\qbar_\Sigma^+ - \qbar_\Sigma\right)^2}{2}f''\left(\qbar_\Sigma\right) + \frac{\left(\qbar_\Sigma^+ - \qbar_\Sigma\right)^3}{6}f''\left(\xi\right) \nonumber\\
	&\begin{aligned}\label{eq:stein-Taylor-f}
	 & = f\left(\qbar_\Sigma\right) + \dfrac{2N^{-\alpha}}{\sigma_a^2+\sigma_s^2}\left(\abar-\sum_{i=1}^n \sbar_i + \sum_{i=1}^N \ubar_i\right)f'\left(\qbar_\Sigma \right)  \\
	 &\quad + \dfrac{2N^{-2\alpha}}{\left(\sigma_a^2+\sigma_s^2\right)^2} \left(\abar-\sum_{i=1}^N \sbar_i+\sum_{i=1}^N \ubar_i \right)^2 f''\left(\qbar_\Sigma\right) \\
	 &\quad + \dfrac{4N^{-3\alpha}}{3\left(\sigma_a^2+\sigma_s^2\right)^3} \left(\abar-\sum_{i=1}^N \sbar_i+\sum_{i=1}^N \ubar_i \right)^3 f''\left(\xi \right) ,
	\end{aligned}
	\end{align}
	where we used that, by definition of $\vqbar^+$ and because $\abar=\sum_{i=1}^N \abar_i$, 
	\begin{align*}
	\qbar_\Sigma^+ -\qbar_\Sigma = \dfrac{2N^{-\alpha}}{\sigma_a^2+\sigma_s^2}\left(\abar - \sum_{i=1}^N \sbar_i + \sum_{i=1}^N \ubar_i\right).
	\end{align*}
	We take expectation of \eqref{eq:stein-Taylor-f} with respect to the stationary distribution. Then, we reorganize terms using the following properties: (i) $f$ is integrable, so set $\E{f\left(\qbar_\Sigma^+\right) }= \E{f\left(\qbar_\Sigma\right)}$; (ii) $\abar$ and $\vsbar$ are independent of $\vqbar$, so $\E{\left(\abar-\sum_{i=1}^N \sbar_i\right)f(\qbar_\Sigma)} = \E{\abar-\sum_{i=1}^N \sbar_i} \E{f(\qbar_\Sigma)}$; (iii) $\E{\sum_{i=1}^N \sbar_i-\abar}=N^{1-\alpha}$; and (iv) the definition of variance. We obtain
	\begin{align}
	\begin{aligned}\label{eq:stein-E-f}
	\E{f'\left(\qbar_\Sigma\right)} &= N^{\alpha-1} \E{\left(\sum_{i=1}^N \ubar_i\right) f'\left(\qbar_\Sigma\right)} + \left(1 + \dfrac{N^{1-2\alpha}}{\sigma_a^2+\sigma_s^2}\right) \E{f''(\qbar_\Sigma)} \\
	&\quad + \dfrac{1}{N\left(\sigma_a^2+\sigma_s^2\right)}\E{\left(\sum_{i=1}^N \ubar_i\right)^2 f''(\qbar_\Sigma)} \\
	&\quad + \dfrac{2}{N\left(\sigma_a^2+\sigma_s^2\right)}\E{\left(\abar-\sum_{i=1}^N \sbar_i\right)\left(\sum_{i=1}^N\ubar_i\right)f''(\qbar_\Sigma)} \\
	&\quad + \dfrac{2}{3 N^{1+3\alpha}(\sigma_a^2+\sigma_s^2)^2} \E{\left(\abar - \sum_{i=1}^N \sbar_i + \sum_{i=1}^N \ubar_i\right)^3 f'''(\xi)}
	\end{aligned}
	\end{align}
	
	Then, using \eqref{eq:stein-E-f} and the triangle inequality, we obtain
	\begin{align}
	\begin{aligned}\label{eq:stein-to-bound}
	& \left|\E{f'\left(\qbar_\Sigma\right)-f''\left(\qbar_\Sigma \right)} \right| \\
	& \leq N^{\alpha-1}\left|\E{\left(\sum_{i=1}^N \ubar_i\right)f'\left(\qbar_\Sigma\right)} \right| + \dfrac{N^{1-2\alpha}}{\sigma_a^2+\sigma_s^2}\left|\E{f''(\xi)} \right| \\
	&\quad + \dfrac{1}{N\left(\sigma_a^2+\sigma_s^2\right)} \left|\E{\left(\sum_{i=1}^N \ubar_i\right)^2 f''\left(\xi\right)} \right| \\
	&\quad + \dfrac{2}{N\left(\sigma_a^2 + \sigma_s^2 \right)}\left|\E{\left(\abar-\sum_{i=1}^N\sbar_i\right)\left(\sum_{i=1}^N \ubar_i\right) f''(\xi)} \right| \\
	&\quad + \dfrac{2}{3 N^{1+3\alpha}(\sigma_a^2+\sigma_s^2)^2} \left|\E{\left(\abar - \sum_{i=1}^N \sbar_i + \sum_{i=1}^N \ubar_i\right)^3 f'''(\xi)}\right|.
	\end{aligned}
	\end{align}
	
	We bound each of the terms of \eqref{eq:stein-to-bound}. For the first term, we first expand $f'(\qbar_\Sigma)$ in Taylor series around 0. There exists $\zeta$ between 0 and $\qbar_\Sigma$ such that
	\begin{align*}
	f'\left(\qbar_\Sigma\right) & \stackrel{(a)}{=} \qbar_\Sigma f''(\zeta) \stackrel{(b)}{\leq} \qbar_\Sigma
	\end{align*}
	where $(a)$ holds because $f'(0)=0$ since $f'\in\cF_W$; and $(b)$ holds because $f''(\zeta)\leq 1$ and $\qbar_\Sigma$ is nonnegative. Then, for the first term of \eqref{eq:stein-to-bound} we have
	\begin{align*}
	N^{\alpha-1}\left|\E{\left(\sum_{i=1}^N \ubar_i\right)f'\left(\qbar_\Sigma\right)} \right| & \leq N^{\alpha-1}\E{\left(\sum_{i=1}^N \ubar_i\right)\left|f'\left(\qbar_\Sigma\right) \right|} \\
	& \leq N^{\alpha-1}\E{\left(\sum_{i=1}^N \ubar_i\right)\qbar_\Sigma} \\
	& \stackrel{(a)}{=} \dfrac{2}{N\left(\sigma_a^2+\sigma_s^2\right)}\E{\left(\sum_{i=1}^N \ubar_i\right)\left(\sum_{i=1}^N\qbar_i \right)} \\
	& \stackrel{(b)}{=} \dfrac{2}{\sigma_a^2+\sigma_s^2}\E{\langle \vqbar_\parallel,\vubar_\parallel\rangle} \\
	& \stackrel{(c)}{\leq} \dfrac{2}{\sigma_a^2+\sigma_s^2}\left(\smax N^{1-\alpha} + C \smax(\alpha-1) N^{4-\alpha} \left\lceil\log\left(N\right)\right\rceil N^{\frac{1}{\left\lceil\log\left(N\right)\right\rceil}} \right),
	\end{align*}
	where $(a)$ holds by definition of $\qbar_\Sigma$; $(b)$ holds by definition of $\vqbar_\parallel$ and $\vubar_\parallel$ according to \eqref{eq:def-parallel}; and $(c)$ holds by the claim below.
	
	\begin{claim}
		Consider a load balancing system as described in Theorem \ref{thm:Stein}. Then, 
		\begin{align}\label{claim:stein-qperp}
		\E{\langle\vqbar_\parallel,\vubar_\parallel\rangle} & \leq \smax N^{1-\alpha} + C \smax \left\lceil\alpha-1\right\rceil N^{4-\alpha} \left\lceil\log\left(N\right)\right\rceil N^{\frac{1}{\left\lceil\log\left(N\right)\right\rceil}}.
		\end{align}
	\end{claim} 

	We prove the claim in Appendix \ref{app:claim-stein}. For the second term of \eqref{eq:stein-to-bound} we have
	\begin{align*}
	\dfrac{N^{1-2\alpha}}{\sigma_a^2+\sigma_s^2}\left|\E{f''(\qbar_\Sigma)} \right| & \leq \dfrac{N^{1-2\alpha}}{\sigma_a^2+\sigma_s^2},
	\end{align*}
	because $f'\in \cF_W$ and, therefore, $|f''(\qbar_\Sigma)|\leq 1$ with probability 1.
	
	For the third term we use that $f'\in\cF_W$, that $\ubar_i\leq\smax$ for all $i\in[N]$ and Lemma \ref{lemma:Eu}. We obtain,
	\begin{align*}
	\dfrac{1}{N\left(\sigma_a^2+\sigma_s^2\right)} \left|\E{\left(\sum_{i=1}^N \ubar_i\right)^2 f''\left(\qbar_\Sigma\right)} \right| & \leq \dfrac{\smax N^{1-\alpha}}{\sigma_a^2+\sigma_s^2}.
	\end{align*}
	Similarly, for the fourth term we obtain
	\begin{align*}
	\dfrac{2}{N\left(\sigma_a^2 + \sigma_s^2 \right)}\left|\E{\left(\abar-\sum_{i=1}^N\sbar_i\right)\left(\sum_{i=1}^N \ubar_i\right) f''(\qbar_\Sigma)} \right|\leq \dfrac{2\smax N^{1-\alpha}}{\sigma_a^2+\sigma_s^2}.
	\end{align*}

	Finally, for the last term we obtain
	\begin{align*}
		& \dfrac{2}{3 N^{1+3\alpha}(\sigma_a^2+\sigma_s^2)^2} \left|\E{\left(\abar - \sum_{i=1}^N \sbar_i + \sum_{i=1}^N \ubar_i\right)^3 f'''(\xi)}\right| \\
		&\leq \dfrac{2}{3 N^{1+3\alpha}(\sigma_a^2+\sigma_s^2)^2} \E{\left|\abar - \sum_{i=1}^N \sbar_i + \sum_{i=1}^N \ubar_i\right|^3 |f'''(\xi)|} \\
		&\stackrel{(a)}{\leq} \dfrac{4}{3 N^{1+3\alpha}(\sigma_a^2+\sigma_s^2)^2} \E{\left|\abar - \sum_{i=1}^N \sbar_i + \sum_{i=1}^N \ubar_i\right|^3} \\
		&\stackrel{(b)}{\leq} \dfrac{4N^{2-3\alpha}}{3(\sigma_a^2+\sigma_s^2)^2} \left(\amax + 2\smax\right)^3
	\end{align*}
	where $(a)$ holds because $f'\in\cF_W$ and, hence, $f'''(\xi)|\leq 2$; and $(b)$ holds because $0\leq \ubar_i\leq \sbar_i\leq \smax$ and $\abar\leq N\amax$ with probability 1.
	
	Putting everything together, we obtain
	\begin{align*}
	& \left|\E{f'\left(\qbar_\Sigma \right) - f''\left(\qbar_\Sigma\right) }  \right| \\
	& \leq \dfrac{1}{\sigma_a^2+\sigma_s^2}\left(5\smax N^{1-\alpha} + N^{1-2\alpha} + C \smax(\alpha-1) N^{4-\alpha} \left\lceil\log\left(N\right)\right\rceil N^{\frac{1}{\left\lceil\log\left(N\right)\right\rceil}} \right. \\
	&\qquad \qquad \qquad \left. +  \dfrac{4N^{2-3\alpha}}{3(\sigma_a^2+\sigma_s^2)} \left(\amax + 2\smax\right)^3 \right),
	\end{align*}
	which proves the theorem.  
\Halmos \endproof

%%%%%%%%%%%%%%%%%%%%%%%%%%%%%%%%%%%%%%%%%%%%%%%%%%%%%%%%%%%%%%%%%%%%%%%%%
%%%%%%%%%%%%%%%%%%%%%%%%%%%%%%%%%%%%%%%%%%%%%%%%%%%%%%%%%%%%%%%%%%%%%%%%%
\section{Conclusion and future work}\label{sec:conclusion}

In this paper we study the load balancing system operating under JSQ in the many-server heavy-traffic regime. We parametrize the arrival rate so that the arrival rate \emph{per server} is $N^{-\alpha}$, for $\alpha>0$. Specifically, we answer the question: how fast should the number of servers grow with respect to the load to observe the classical heavy-traffic behavior of the average queue lengths? We show that $\alpha>4$ is sufficient to satisfy this condition. We use two proof techniques: one based on the Transform method proposed in \cite{Hurtado_transform_method} and the other one based on Stein's method. 

The case of $\alpha\leq 1$ is well studied in the literature. Then, there is a gap between our results and the literature. Future work is to explore how the system behaves if $\alpha\in(1,4]$. It is possible that our bounds are not tight and the same classical heavy-traffic behavior can be proved in these intervals. 
In both of our proofs, the value of $\alpha$ is determined so that the terms related to SSC vanish in the limit. One line of future work is to study if it is possible to obtain tighter SSC bounds, so that we can improve the value of $\alpha$. 
However, there may also be another phase transition in these intervals, in which case a different proof technique might be required. 

%%%%%%%%%%%%%%%%%%%%%%%%%%%%%%%%%%%%%%%%%%%%%%%%%%%%%%%%%%%%%%%%%%%%%%%%%%%%%%%%%%%%%%%%%%%%%
%%%%%%%%%%%%%%%%%%%%%%%%%%%%%%%%%%%%%%%%%%%%%%%%%%%%%%%%%%%%%%%%%%%%%%%%%%%%%%%%%%%%%%%%%%%%%
%%%%%%%%%%%%%%%%%%%%%%%%%%%%%%%%%%%%%%%%%%%%%%%%%%%%%%%%%%%%%%%%%%%%%%%%%%%%%%%%%%%%%%%%%%%%%

\newpage 

\renewcommand{\theHsection}{A\arabic{section}}
\begin{APPENDICES}

\section{Moment bounds based on drift arguments}\label{app:bertsimas_lemma}

The proof of Proposition \ref{prop:ssc} is based on the following lemma. 

\begin{lemma}\label{lemma:bertsimas}
	Consider an irreducible and aperiodic Markov chain $\left\{X(k):k\geq 1 \right\}$ over a countable state space $\cX$. Let $Z:\cX\to \bR_+$ be a nonnegative valued Lyapunov function, and suppose its drift satisfies the following conditions
	\begin{enumerate}[label=(C\arabic*)]
		\item There exists $\eta>0$ and $\kappa<\infty$ such that for any $k\in\bZ_+$ and for all $x\in \cX$ that satisfies $Z(x)\geq \kappa$,
		\begin{align*}
		\E{\left. \Delta Z(x)\right| X(k)=x}\leq -\eta
		\end{align*}
		
		\item There exists $D<\infty$ such that for all $x\in\cX$,
		\begin{align*}
		\Prob{|\Delta Z(x)|\leq D }=1.
		\end{align*}
	\end{enumerate}
	Further, assume that the Markov chain $\left\{X(k):k\in\bZ_+ \right\}$ converges in distribution to a random variable $\Xbar$. Then, for any $r\in\bZ_+$ with $r\geq 1$,
	\begin{align}\label{eq:hajek.moments}
	\E{Z(\Xbar)^r}\leq (2\kappa)^r + (4D)^r\left(\dfrac{D+\eta}{\eta}\right)^r r!.
	\end{align}

Additionally, for any $\theta^* \leq \frac{1}{2D}\log\left(1+\tfrac{\eta}{D}\right)$ we have
\begin{align}\label{eq:hajek.mgf}
	\E{e^{\theta^* Z(\Xbar)}}\leq e^{\theta^*\kappa }\left(\frac{\eta}{\eta + D (1-e^{2\theta^*D})}\right).
\end{align}
\end{lemma}

Equation \eqref{eq:hajek.moments} was proved in \cite[Lemma 3]{MagSri_SSY16_Switch}. The proof of \eqref{eq:hajek.mgf} follows similar steps, and we omit it for brevity.

Now we present the proof of Proposition \ref{prop:ssc}.
\proof{Proof of Proposition \ref{prop:ssc}.}
We use the Lyapunov function $W_\perp(\vq)=\left\|\vq_\perp\right\|$. For ease of exposition we omit the superscript $\pN$ of the variables.

We first prove that condition $(C2)$ in Lemma \ref{lemma:bertsimas} is satisfied. We have
\begin{align*}
	\left|\Delta W_\perp (\vq) \right| &= \left|\left\|\vq_\perp(k+1)\right\|- \left\|\vq_\perp(k)\right\| \right| \ind{\vq(k)=\vq} \\
	& \stackrel{(a)}{\leq} \left\|\vq_\perp(k+1)-\vq_\perp(k) \right\| \ind{\vq(k)=\vq} \\
	& \stackrel{(b)}{\leq} \left(\left\|\vq(k+1)-\vq(k)\right\| + \|\vq_\parallel(k+1)-\vq_\parallel(k) \| \right) \ind{\vq(k)=\vq} \\
	& \stackrel{(c)}{\leq} 2\left\|\vq(k+1)-\vq(k) \right\| \ind{\vq(k)=\vq} \\
	& \stackrel{(d)}{\leq} 2\sqrt{N}\max\{N\amax,\smax\} \\
	& \stackrel{(e)}{\leq} 2\amax N^{\frac{3}{2}},
\end{align*}
where $(a)$ holds by triangle inequality; $(b)$ holds because $\vq_\perp(k)=\vq(k)-\vq_\parallel(k)$ by definition, and by triangle inequality; $(c)$ holds because the projection is nonexpansive; $(d)$ holds by the dynamics of the queues presented in \eqref{eq:dynamics.max}, and because $a(k)\leq N \amax$ and $s_i(k)\leq \smax$ with probability 1 by assumption; and $(e)$ holds for $N$ large (it suffices $N\geq \frac{\smax}{\amax}$). Therefore, condition $(C2)$ is satisfied with $D = 2\amax N^{\frac{3}{2}}$.

Now we verify condition $(C1)$. We are interested in showing that $\Delta W_\perp(\vq)$ is upper bounded by a negative constant when $W_\perp(\vq) \geq \kappa$, where $\kappa>0$. Then, without loss of generality, we assume $W_\perp(\vq)>0$ in this part of the proof. Before proceeding we introduce the following notation. Let
\begin{align*}
	V(\vq) \defn \|\vq\|^2\quad \text{and}\quad V_\parallel(\vq)=\|\vq_\parallel\|^2,
\end{align*}
and observe that $W_\perp(\vq)=\sqrt{\left\|\vq_\perp \right\|^2}$. Then, since $f(x)=\sqrt{x}$ is a concave function, and by definition of $\vq_\parallel$ and $\vq_\perp$, we have
\begin{align}\label{eq:ssc.concave}
	\Delta W_\perp(\vq)\leq \dfrac{1}{2\|\vq_\perp\|}\left(\Delta V(\vq)-\Delta V_\parallel(\vq) \right).
\end{align}
Equation \eqref{eq:ssc.concave} was first proved in \cite{atilla}. Now we bound $\Eq{V(\vq)}$ and $\Eq{V_\parallel(\vq)}$ separately. We have
\begin{align*}
	& \Eq{\Delta V(\vq)}\\
	& = \Eq{\|\vq(k+1)\|^2- \|\vq(k)\|^2} \\
	& = \Eq{\|\vq(k+1)-\vu(k) + \vu(k)\|^2- \|\vq(k)\|^2} \\
	& \stackrel{(a)}{=} \Eq{\|\vq(k)+\va(k)-\vs(k)\|^2 + \|\vu(k)\|^2 + 2\langle \vq(k+1)-\vu(k),\vu(k)\rangle - \|\vq(k)\|^2 } \\
	& \stackrel{(b)}{=} \Eq{\|\vq(k)+\va(k)-\vs(k) \|^2 -\|\vu(k)\|^2 - \|\vq(k)\|^2} \\
	& \stackrel{(c)}{=} \Eq{\| \va(k)-\vs(k)\|^2 + 2\langle\vq(k),\va(k)-\vs(k)\rangle - \|\vu(k)\|^2} \\
	& \stackrel{(d)}{\leq} \Eq{\| \va(k)-\vs(k)\|^2 + 2\langle\vq(k),\va(k)-\vs(k)\rangle}, \numberthis \label{eq:SSC.Vq.partial}
\end{align*}
where $(a)$ holds by definition of the Euclidean norm and dot product, and by \eqref{eq:dynamics.u}; $(b)$ holds by \eqref{eq:qu} and rearranging terms; $(c)$ holds by definition of the Euclidean norm and rearranging terms; and $(d)$ holds because $\Eq{\|\vu(k)\|^2}\geq 0$.

Now we bound each of the terms. We start bounding $\E{\left\|\va(k)-\vs(k)\right\|^2}$. We have
\begin{align*}
	\Eq{\|\va(k)-\vs(k)\|^2} & \stackrel{(a)}{\leq} \Eq{\left\| \va(k)\right\|^2} + \E{\left\|\vs(k)\right\|^2} \\
	&\stackrel{(b)}{=} \E{a^2(k)} + \sum_{i=1}^n \E{s_i^2(k)} \\
	&\stackrel{(c)}{=} N^2(1-N^{-\alpha})^2 + N\sigma_a^2 + N + N\sigma_s^2 \\
	&\stackrel{(d)}{\leq} 2 N^2, \numberthis \label{eq:SSC.Vq.partial1}
\end{align*}
where $(a)$ holds expanding the square and because $\langle\va(k),\vs(k)\rangle\geq 0$; $(b)$ holds because all the arrivals in one time slot are routed to the same queue, and by definition of the Euclidean norm; $(c)$ holds by definition of variance; and $(d)$ holds for $N$ large (it suffices to take $N\geq \sigma_a^2+\sigma_s^2+1$). 

Now we bound the second term of \eqref{eq:SSC.Vq.partial}. Let $i^*\in\argmin_{i\in[N]} q_i$ be the queue where the arrivals are routed. Then,
\begin{align*}
	\Eq{\langle\vq,\va(k)-\vs(k)\rangle } &\stackrel{(a)}{=} q_{i^*}N(1-N^{-\alpha}) - \sum_{i=1}^{N} q_i \\
	&\stackrel{(b)}{=} -N^{-\alpha}\sum_{i=1}^N q_i + (1-N^{-\alpha}) \sum_{i=1}^N (q_{i^*}-q_i),% \\
	%&\stackrel{(c)}{\leq} -N^{-\alpha}\sum_{i=1}^N q_i \numberthis \label{eq:SSC.Vq.partial2}
\end{align*}
where $(a)$ holds by definition of JSQ; and $(b)$ holds after adding and subtracting $(1-N^{-\alpha})\sum_{i=1}^N q_i$ and rearranging terms. Also notice that, by definition of $\delta$, we know $\delta \leq 1-N^{-\alpha}$ for all $N\geq N_0$. Then,
\begin{align*}
	(1-N^{-\alpha})\sum_{i=1}^N (q_{i^*}-q_i) & \leq -\delta \sum_{i=1}^N |q_{i^*}-q_i| \\
	&\stackrel{(a)}{\leq} -\delta \left\| \vq-q_{i^*}\vone\right\| \\
	&\stackrel{(b)}{\leq} -\delta \left\|\vq_\perp \right\|,
\end{align*}
where $(a)$ holds because norm-1 upper bounds the Euclidean norm; and $(b)$ holds because $\frac{1}{N}\sum_{i=1}^N q_i$ is the minimizer of the function $f(x)=\left\|\vq-x\vone\right\|$ by definition of projection, and because $\vq_\perp=\vq-\vq_\parallel$. Then,
\begin{align}\label{eq:SSC.Vq.partial2}
	\Eq{\langle \vq,\va(k)-\vs(k)\rangle} \leq -N^{-\alpha}\sum_{i=1}^N q_i -\delta \left\|\vq_\perp \right\|.
\end{align}

Then, using \eqref{eq:SSC.Vq.partial1} and \eqref{eq:SSC.Vq.partial2} in \eqref{eq:SSC.Vq.partial} we obtain
\begin{align}\label{eq:SSC.Vq}
	\Eq{\Delta V(\vq)} \leq 2N^2 -2\delta\left\|\vq_\perp\right\| - 2N^{-\alpha} \sum_{i=1}^{N} q_i.
\end{align}

Now we bound $\Eq{\Delta V_\parallel(\vq)}$. In the computation of this bound we follow similar steps to \cite{atilla} and \cite{MagSri_SSY16_Switch} in their proof of SSC. We present the complete proof for completeness. We have
\begin{align}
	\Eq{\Delta V_\parallel(\vq)} &= \Eq{\|\vq_\parallel(k+1)\|^2 - \|\vq_\parallel(k)\|^2} \nonumber \\
	& = \Eq{\langle \vq_\parallel(k+1)+\vq_\parallel(k), \vq_\parallel(k+1)-\vq_\parallel(k) \rangle} \nonumber \\
	& \stackrel{(a)}{=} \Eq{\|\vq_\parallel(k+1)-\vq_\parallel(k) \|^2} + 2\Eq{\langle\vq_\parallel(k), \vq_\parallel(k+1)-\vq_\parallel(k)\rangle} \nonumber \\
	& \stackrel{(b)}{\geq} 2\Eq{\langle\vq_\parallel(k), \vq_\parallel(k+1)-\vq_\parallel(k)\rangle} \nonumber \\
	& \stackrel{(c)}{=} 2\Eq{\langle \vq_\parallel(k), \vq(k+1)-\vq(k)\rangle} \nonumber \\
	& \stackrel{(d)}{\geq} 2\Eq{\langle\vq_\parallel(k),\va(k)-\vs(k) \rangle} \nonumber \\
	& \stackrel{(e)}{=} -2\left(\sum_{i=1}^N \frac{q_i}{N}\right)N^{1-\alpha} \nonumber \\
	& = -2 N^{-\alpha}\sum_{i=1}^N q_i, \label{eq:ssc.V_parallel-bound}
\end{align}
where $(a)$ holds by definition of the Euclidean norm; $(b)$ holds because the norm of any vector is nonnegative; $(c)$ holds because $\vq_\parallel(k)$ is orthogonal to $\vq_\perp(k+1)$ and to $\vq_\perp(k)$ by definition; $(d)$ holds by the dynamics of the queues presented in \eqref{eq:dynamics.u} and because, by definition, $\vq_\parallel(k)\geq \vzero$ and $\vu(k)\geq 0$ so their dot product is nonnegative; and $(e)$ holds because, by definition, $\qbar_{\parallel i}=\frac{1}{N}\sum_{j=1}^N \qbar_j$ for all $i\in[N]$ and because the arrival and service processes are independent of the queue lengths.

Then, using \eqref{eq:SSC.Vq} and \eqref{eq:ssc.V_parallel-bound} in \eqref{eq:ssc.concave} we obtain
\begin{align*}
	\Eq{W_\perp(\vq)} &\leq \dfrac{1}{2\left\|\vq_\perp\right\|}\left(2N^2 - \delta \left\|\vq_\perp\right\| \right) \\
	&= \dfrac{N^2}{\left\|\vq_\perp\right\|} - \delta.
\end{align*}
Hence, condition $(C1)$ is satisfied with
\begin{align*}
	\eta \defn\frac{\delta}{2} ,\quad \kappa\defn \dfrac{2 N^2}{\delta}
\end{align*}
for all $N\geq \tilde{N}$, with
\begin{align}\label{eq:SSC.N0}
	\tilde{N}\defn \max\left\{\frac{\smax}{\amax},\, \sigma_a^2+\sigma_s^2+1 \right\}.
\end{align}
Therefore, using \eqref{eq:hajek.moments} from Lemma \ref{lemma:bertsimas} we obtain that for each $r\in\bZ_+$ with $r\geq 1$ we have
\begin{align*}
	\E{\left\| \vq_\perp\right\|^r} &\leq \left(\dfrac{4N^2}{\delta}\right)^r + \left(2\amax N^{\frac{3}{2}}\right)^r\left(\dfrac{4\amax N^{\frac{3}{2}} + \delta}{\delta} \right)^r r! \\
	&\stackrel{(a)}{\leq}\tilde{C}^r N^{3r} r! \\
	&\stackrel{(b)}{\leq} \tilde{C}^r N^{3r} e^{1-r} r^{r+\frac{1}{2}},
\end{align*}
where $(a)$ holds for a finite constant $\tilde{C}$ that only depends on $\amax$ and $\delta$; and $(b)$ holds by Stirling's inequality for the factorial. Then, we get
\begin{align*}
	\Eq{\left\| \vq_\perp\right\|^r}^{\frac{1}{r}}\leq C N^3 r,
\end{align*}
where $C\defn \tilde{C} e^e$. This proves \eqref{eq:SSC.moments}. To compute \eqref{eq:SSC.mgf} we use \eqref{eq:hajek.mgf}. 
%	\begin{align*}
%		\E{\left\|\vqbar_\perp\right\|^r} &\leq \left(\dfrac{4 N^{\frac{7}{2}}}{(1-\delta)(d-1)} \right)^r + \left(8\amax N^{\frac{3}{2}}\right)^r \left( \dfrac{4\amax N^3 + (1-\delta)(d-1)}{(1-\delta)(d-1)} \right)^r r! \\
%		&\stackrel{(a)}{\leq} \tilde{C}^r \left(\dfrac{N^{\frac{9}{2}}}{d-1}\right)^r r! \\
%		&\stackrel{(b)}{\leq} \tilde{C}^r \left(\dfrac{N^{\frac{9}{2}}}{d-1}\right)^r e^{1-r} r^{r+\frac{1}{2}}
%	\end{align*}
%	where $(a)$ holds for a finite constant $\tilde{C}$ that depends on $\delta$ and $\amax$; and $(b)$ holds by Stirling's inequality for the factorial. Then,
%	\begin{align*}
%		\E{\left\|\vq_\perp\right\|^r}^{\frac{1}{r}} &\leq \tilde{C}\left(\dfrac{N^{\frac{9}{2}}}{d-1}\right) e^{\frac{1}{r}-1} r^{1+\frac{1}{2r}} \\
%		&\stackrel{(a)}{\leq} C \left(\dfrac{N^{\frac{9}{2}}}{d-1}\right)r,
%	\end{align*}
%	where $(a)$ holds for $C\defn \tilde{C} e^e$ because $e^{\frac{1}{r}-1}\leq 1$ and $r^{\frac{1}{2r}}\leq e^e$.
\Halmos \endproof

\section{Details of proofs of Section \ref{sec:MGF}}

\subsection{Proof of Lemma \ref{lemma:expo_qu}}\label{app:lemma-expo-qu}

We use the following lemma, which was proved in \cite[Lemma 12]{Hurtado_transform_method}. We state it here for completeness.

\begin{lemma}\label{lemma:qperp-u}
	Consider a load balancing system operating under JSQ, parametrized by $N$ as described in Section \ref{sec:model}. Then, for any $\zeta\in\bR$ and $k\in\bZ_+$ we have
	\begin{align*}
	\sum_{i=1}^N u_i^\pN(k)\left(\expo{\frac{\zeta}{N}\sum_{j=1}^N q_j^\pN(k+1)}-1\right) = \sum_{i=1}^N u_i^\pN(k)\left(\expo{-\zeta q_{\perp i}^\pN(k+1)}-1\right),
	\end{align*} 
	where $q_{\perp i}^\pN(k+1)$ is the $i\tth$ element of $\vq_\perp^\pN(k+1)$.
\end{lemma}

Now we prove Lemma \ref{lemma:expo_qu}.

\proof{Proof of Lemma \ref{lemma:expo_qu}.}
	We have
	\begin{align}
	& \left|\E{\left(\expo{\thetahat N^{-\alpha} \sum_{i=1}^N \qbar_i^+} -1\right)\left(\expo{-\thetahat N^{-\alpha}\sum_{i=1}^N \ubar_i} -1\right)} \right| \nonumber \\
	& \leq \E{\left|\left(\expo{\thetahat N^{-\alpha} \sum_{i=1}^N \qbar_i^+} -1\right)\left(\expo{-\thetahat N^{-\alpha}\sum_{i=1}^N \ubar_i} -1\right) \right|} \nonumber \\
	& \stackrel{(a)}{=} |\thetahat| N^{-\alpha}\E{\left(\sum_{i=1}^N \ubar_i\right)\left|\expo{\thetahat N^{-\alpha} \sum_{i=1}^N \qbar_i^+} -1 \right| \left(\dfrac{\expo{-\thetahat N^{-\alpha}\sum_{i=1}^N \ubar_i} -1}{-\thetahat N^{-\alpha}\sum_{i=1}^N \ubar_i}\right)\ind{\sum_{i=1}^N \ubar_i\neq 0} } \nonumber \\
	& \stackrel{(b)}{\leq} |\thetahat| N^{-\alpha} \left(\dfrac{\expo{-\thetahat N^{-\alpha}\smax}-1}{-\thetahat N^{-\alpha} \smax} \right) \E{\sum_{i=1}^N \ubar_i\left|\expo{\thetahat N^{-\alpha} \sum_{j=1}^N \qbar_j^+}-1 \right|} \nonumber \\
	& \stackrel{(c)}{\leq} |\thetahat| N^{-\alpha} \left(\dfrac{\expo{-\thetahat N^{-\alpha}\smax}-1}{-\thetahat N^{-\alpha} \smax} \right) \E{\sum_{i=1}^N \ubar_i\left|\expo{-\thetahat N^{1-
				\alpha}\qbar_{\perp i}}-1 \right|} \nonumber \\
	& \stackrel{(d)}{\leq}  |\thetahat| N^{-\alpha} \left(\dfrac{\expo{-\thetahat N^{-\alpha}\smax}-1}{-\thetahat N^{-\alpha} \smax} \right) \E{\sum_{i=1}^N \ubar_i^p}^{\frac{1}{p}} \E{\sum_{i=1}^N \left|\expo{-\thetahat N^{1-\alpha}\qbar_{\perp i}}-1 \right|^{\frac{p}{p-1}}}^{\frac{p-1}{p}} \label{eq:lemma.expo.qu.partial}
	\end{align}
	where $p>1$ is an integer number. Here $(a)$ holds after multiplying and dividing by $|\thetahat| N^{-\alpha}\sum_{i=1}^N\ubar_i$, because if $\sum_{i=1}^N \ubar_i=0$ then $\expo{\thetahat N^{-\alpha}\sum_{i=1}^N \ubar_i} -1=0$, and because the function $f(x)=\frac{e^x-1}{x}$ is nonnegative; $(b)$ holds because the function $f(x)=\frac{e^x-1}{x}$ is nonnegative and increasing, and because $0\leq\ubar_i\leq\smax$ with probability 1 for all $i\in[N]$; $(c)$ holds by Lemma \ref{lemma:qperp-u} and the triangle inequality; and $(d)$ holds by Hölder's inequality.
	
	We analyze each expression in \eqref{eq:lemma.expo.qu.partial} separately. First observe that
	\begin{align*}
	\lim_{N\to\infty}\dfrac{\expo{-\thetahat N^{-\alpha}\smax}-1}{-\thetahat N^{-\alpha} \smax} = 1.
	\end{align*}
	
	The unused service is nonnegative by definition, then
	\begin{align*}
	0\leq \E{\sum_{i=1}^N \ubar_i^p} & \stackrel{(a)}{\leq} \smax^{p-1}\,\E{\sum_{i=1}^N \ubar_i}  \stackrel{(b)}{=} \smax^{p-1}N^{1-\alpha}
	\end{align*}
	where $(a)$ holds because $\ubar_i\leq \smax$ with probability 1 for all $i\in[N]$; and $(b)$ holds by Lemma \ref{lemma:Eu}.
	
	For the last term we use Hölder's inequality again. Let $r>1$. Then, for each $i\in[N]$ we have
	\begin{align*}
	& \E{\left|\expo{-\thetahat N^{1-\alpha}\qbar_{\perp i}}-1 \right|^{\frac{p}{p-1}}} \\
	& = |\thetahat|^{\frac{p}{p-1}}N^{\frac{p}{p-1}(1-\alpha)} \E{\left(\dfrac{\expo{-\thetahat N^{1-\alpha}\qbar_{\perp i}}-1}{-\thetahat N^{1-\alpha}\qbar_{\perp i}} \right)^{\frac{p}{p-1}} \left|\qbar_{\perp i}\right|^{\frac{p}{p-1}}\ind{\qbar_{\perp i}\neq 0}} \\
	& \stackrel{(a)}{\leq} \thetahat^{\frac{p}{p-1}}N^{\frac{p}{p-1}(1-\alpha)} \left(\E{\left|\dfrac{\expo{-\thetahat N^{1-\alpha} \qbar_{\perp i}}-1}{-\thetahat N^{1-\alpha} \qbar_{\perp i}} \right|^{\left(\frac{p}{p-1}\right)\left(\frac{r}{r-1}\right)}\ind{\qbar_{\perp i}\neq 0}}\right)^{\frac{r-1}{r}} \left(\E{|\qbar_{\perp i}|^{\frac{p}{p-1}r}} \right)^{\frac{1}{r}},
	\end{align*}
	where $(a)$ holds by Hölder's inequality. We bound each of these terms.
	
	Using Proposition \ref{prop:ssc} for the last term we obtain
	\begin{align}\label{eq:bound-SSC}
	0\leq \E{|\qbar_{\perp i}|^{\frac{p}{p-1}r}}^{\frac{p-1}{pr}} \stackrel{(a)}{\leq} \E{\|\vqbar_\perp\|^{\frac{p}{p-1}r}}^{\frac{p-1}{pr}}\leq CN^3 \left(\dfrac{pr}{p-1}\right) ,
	\end{align}
	where $(a)$ holds if $\frac{p}{p-1}r\geq 2$, by the inequalities between norms. 
	
	On the other hand, since $\frac{e^x-1}{x}\leq e^{|x|}$, we obtain
	\begin{align}\label{eq:SSC-expo-bound-proof}
		\left(\E{\left|\dfrac{\expo{-\thetahat N^{1-\alpha} \qbar_{\perp i}}-1}{-\thetahat N^{1-\alpha} \qbar_{\perp i}} \right|^{\left(\frac{p}{p-1}\right)\left(\frac{r}{r-1}\right)}\ind{\qbar_{\perp i}\neq 0}}\right)^{\frac{r-1}{r}} \leq \left(\E{\expo{|\thetahat| N^{1-\alpha} \left(\frac{r-1}{r}\right)\left(\frac{p-1}{p}\right) |\qbar_{\perp i}|}} \right)^{\frac{r-1}{r}},
	\end{align}
	and the last term can be bounded similarly to \eqref{eq:bound-SSC}, using the second part of Proposition \ref{prop:ssc} for 
	\begin{align*}
		|\thetahat| \left(\frac{r-1}{r}\right)\left(\frac{p-1}{p}\right) \leq \frac{1}{4\amax N^{\frac{5}{2}-\alpha}} \log\left(1+\frac{\delta}{4\amax N^{\frac{3}{2}}} \right).
	\end{align*}
	Observe that the right-hand side grows to infinity as $N\to\infty$. Then, there exists $N_0^*\in\bZ_+$ such that the lemma is satisfied with
	\begin{align*}
		\tilde{\Theta} \defn \left(\frac{r}{r-1}\right)\left(\frac{p}{p-1}\right) \left(\frac{1}{4\amax (N_0^*)^{\frac{5}{2}-\alpha}}\right) \log\left(1+\frac{\delta}{4\amax (N_0^*)^{\frac{3}{2}}} \right).
	\end{align*}
	Further, the upper bound in \eqref{eq:SSC-expo-bound-proof} converges to a constant as $N\to\infty$. We omit the details for brevity.
	
	Putting everything together in \eqref{eq:lemma.expo.qu.partial} we obtain
	\begin{align*}
	\left|\E{\left(\expo{\thetahat N^{-\alpha} \sum_{i=1}^N \qbar_i^+} -1\right)\left(\expo{-\thetahat N^{-\alpha}\sum_{i=1}^N \ubar_i} -1\right)} \right|\leq L(N)\, N^{4-2\alpha + \frac{1-\alpha}{p}},
	\end{align*}
	where $L(N)$ is of order $O(1)$. Finally, observe that $4-2\alpha + \frac{1-\alpha}{p}<1-2\alpha$ if and only if $\alpha>3p+1$, and $p$ can be taken as close to one as desired. Therefore, the lemma holds for all $\alpha>4$.  
\Halmos \endproof

\section{Details of proofs in Section \ref{sec:Stein}.}

\subsection{Proof of \eqref{claim:stein-qperp}} \label{app:claim-stein}

\proof{Proof of \eqref{claim:stein-qperp}.}
	We need to find an upper bound for $\E{\langle\vqbar_\parallel,\vubar_\parallel\rangle}$. By definition of $\vqbar_\parallel$ and $\vu_\parallel$ we have
	\begin{align}
	\langle\vqbar_\parallel,\vubar_\parallel\rangle & \stackrel{(a)}{=} \langle \vqbar_\parallel,\vubar\rangle \nonumber\\
	& \stackrel{(b)}{=} \langle\vqbar,\vubar\rangle - \langle\vqbar_\perp,\vubar\rangle \label{eq:stein-q-perp-0}
	\end{align}
	where $(a)$ holds because $\vqbar_\parallel$ is orthogonal to $\vubar_\perp$; and $(b)$ by the definition of $\vqbar_\perp$ according to \eqref{eq:def-parallel}. We analyze each term separately. Observe that, by definition of dot product, we have
	\begin{align}
	\E{\langle\vqbar,\vubar\rangle} & = \E{\sum_{i=1}^N \qbar_i\ubar_i} \stackrel{(a)}{\leq} \smax\E{\sum_{i=1}^N \ubar_i} \stackrel{(b)}{=} \smax N^{1-\alpha}, \label{eq:stein-q-perp-1}
	\end{align}
	where $(a)$ holds because if $\ubar_i=\upsilon>0$, then by \eqref{eq:dynamics.u} we have $\qbar_i^+=0$, which implies $\qbar_i=\sbar_i-\abar_i-\upsilon\leq \sbar_i\leq \smax$; and $(b)$ holds by Lemma \ref{lemma:Eu}.
	
	To bound the second term we use SSC as proved in Proposition \ref{prop:ssc}. By H\"older's inequality, we obtain
	\begin{align*}
	\left|\E{\langle\vqbar_\perp,\vubar\rangle}\right| & \leq \E{\left\|\vqbar_\perp \right\|_r^r}^{\tfrac{1}{r}} \E{\left\|\vubar \right\|_{r^*}^{r^*}}^{\tfrac{1}{r^*}},
	\end{align*}
	where $r,r^*>1$ and $\frac{1}{r}+\frac{1}{r^*}=1$. On one hand, for $r\geq 2$ we have
	\begin{align*}
	\E{\left\|\vqbar_\perp \right\|_r^r}^{\frac{1}{r}} \leq \E{\left\|\vqbar_\perp \right\|^r}^{\frac{1}{r}} \leq CN^3 r,
	\end{align*}
	where the last inequality holds by Proposition \ref{prop:ssc}. On the other hand, since $r^*>1$ we have
	\begin{align*}
	\E{\left\|\vubar \right\|^{r^*}_{r^*}}^{\frac{1}{r^*}} & = \E{\sum_{i=1}^N \ubar_i^{r^*}}^{\frac{1}{r^*}} \\
	& \stackrel{(a)}{\leq} \smax^{\frac{1}{r}} \E{\sum_{i=1}^N \ubar_i}^{1-\frac{1}{r}} \\
	& \stackrel{(b)}{=} \smax^{\frac{1}{r}} N^{(1-\alpha)\left(1-\frac{1}{r}\right)} \\
	&\stackrel{(c)}{\leq} \smax N^{(1-\alpha)\left(1-\frac{1}{r}\right)} , \\
	\end{align*}
	where $(a)$ holds because $\ubar_i\leq\sbar_i\leq\smax$ for all $i\in[N]$ by definition of unused service and because $\frac{1}{r}+\frac{1}{r^*}=1$; $(b)$ holds by Lemma \ref{lemma:Eu}; and $(c)$ because $\smax\geq 1$ and $r>1$. Therefore, we have
	\begin{align}
		\left|\E{\langle\vqbar_\perp,\vubar \rangle}\right| & \leq C\smax N^{4-\alpha} r N^{\frac{1}{r}(\alpha-1)}\label{eq:stein-bound-q-perp-r}.
	\end{align}

	Then, we minimize the upper bound \eqref{eq:stein-bound-q-perp-r} with respect to $r$ and we obtain that the minimizer is $r=(\alpha-1)\log\left(N\right)$. However, the value of $r$ must be integer. Therefore, we use $r=\left\lceil\alpha-1\right\rceil\left\lceil\log\left(N\right)\right\rceil$. Replacing this result in \eqref{eq:stein-bound-q-perp-r} we obtain
	\begin{align}\label{eq:stein-q-perp-2}
		\left|\E{\langle\vqbar_\perp,\vubar \rangle}\right| & \leq C \smax\left\lceil\alpha-1\right\rceil N^{4-\alpha} \left\lceil\log\left(N\right)\right\rceil N^{\frac{1}{\left\lceil\log\left(N\right)\right\rceil}}.
	\end{align}
	Putting \eqref{eq:stein-q-perp-1} and \eqref{eq:stein-q-perp-2} together we obtain the result. 
\Halmos \endproof

\end{APPENDICES}
%\begin{acknowledgements}
%\end{acknowledgements}


\bibliographystyle{informs2014}
%\bibliography{../../biblio-ok}   % name your BibTeX data base
\bibliography{biblio-ok}

\end{document}